\documentclass[10pt]{article}
\usepackage{amssymb}
\usepackage{amsmath}
\usepackage{amsthm}
\newtheorem{theorem}{Theorem}[section]
\newtheorem{conjecture}[theorem]{Conjecture}
\newtheorem{lemma}[theorem]{Lemma}
\newtheorem{proposition}[theorem]{Proposition}
\newtheorem{corollary}[theorem]{Corollary}

\newcommand{\RR}{\mathbb{R}}
\newcommand{\CC}{\mathbb{C}}
\newcommand{\KK}{\mathbb{K}}
\newcommand{\NN}{\mathbb{N}}
\newcommand{\ZZ}{\mathbb{Z}}
\newcommand{\vvv}{|\!|\!|}
\newcommand{\bigvvv}{\Big|\!\Big|\!\Big|}

\newcommand{\Mat}{\mathrm{Mat}}

\begin{document}
\title{Criteria for the stability of the finiteness property and for the uniqueness of Barabanov norms\footnote{This research was funded by EPSRC grant EP/E020801/01.}}\author{Ian D. Morris}
\maketitle

\begin{abstract}
A set of matrices is said to have the \emph{finiteness property} if the maximal rate of exponential growth of long products of matrices drawn from that set is realised by a periodic product. The extent to which the finiteness property is prevalent among finite sets of matrices is the subject of ongoing research. In this article we give a condition on a finite irreducible set of matrices which guarantees that the finiteness property holds not only for that set, but also for all sufficiently nearby sets of equal cardinality. We also prove a theorem giving conditions under which the Barabanov norm associated to a finite irreducible set of matrices is unique up to multiplication by a scalar, and show that in certain cases these conditions are also persistent under small perturbations.
\\
\\
MSC-2010 codes: 15A18, 15A60.
\end{abstract}

\section{Introduction}
Let $\mathcal{A}$ be a bounded set of $d \times d$ real or complex matrices. The joint spectral radius of $\mathcal{A}$, introduced by G.-C. Rota and W. G. Strang in \cite{RS}, is defined to be the quantity
\[\varrho(\mathcal{A}) = \lim_{n \to \infty} \sup\left\{\|A_{i_n}\cdots A_{i_1}\|^{1/n} \colon A_i \in \mathcal{A}\right\},\]
which may easily be shown to yield a finite value which is independent of the choice of norm $\|\cdot\|$. The joint spectral radius has been the subject of substantial recent research interest, which has dealt with its applications \cite{Gu,Koz,MO,DL0,Mae2,GZ}, with issues of its computation and approximation \cite{Koz,Mae2,Koz2,BN,BW,Gr,LW,MQBWF,Prot,BT1}, and with the study of its mathematical properties \cite{Ba,GWZ,Wirth1,Wirth2}. In this article we investigate the \emph{finiteness property}, a property of sets of matrices $\mathcal{A}$ which facilitates the computation of $\varrho(\mathcal{A})$, and the \emph{Barabanov norm}, a theoretical construction associated to a set of matrices $\mathcal{A}$ whose applications have been explored in \cite{Koz2,Ba,Wirth1,Koz3}.
  
Let us begin by establishing some notation. Throughout this article we shall use the symbol $\KK$ to stand for either $\RR$ or $\CC$. Statements which are expressed using $\KK$ should thus be understood as being valid in both the case $\KK=\RR$ and the case $\KK=\CC$. For $d \in \NN$ we let $\Mat_d(\KK)$ denote the ring of $d \times d$ matrices over $\KK$, and let $\vvv \cdot \vvv$ denote the Euclidean norm on $\KK^d$. To simplify the statement of some of our results we shall find it convenient to work primarily with \emph{ordered} sets of elements of $\Mat_d(\KK)$. Thus, for each pair of integers $r,d \in \NN$ we let $\mathcal{O}_r(\KK^d)$ denote the set of all ordered $r$-tuples of elements of $\Mat_d(\KK)$. We equip each $\mathcal{O}_r(\KK^d)$ with the metric
\[d_{\mathcal{O}_r}\left[(A_1,\ldots,A_r),(B_1,\ldots,B_r)\right]:=\max_{1 \leq i \leq r}\vvv A_i-B_i\vvv.\]

An $r$-tuple $\mathcal{A}=(A_1,\ldots,A_r) \in \mathcal{O}_r(\KK^d)$ is said to have the \emph{finiteness property} if there exist $n \in \NN$ and  $(i_1,\ldots,i_n)\in\{1,\ldots,r\}^n$ such that $\rho(A_{i_n}\cdots A_{i_1})^{1/n} =\varrho(\mathcal{A})$. The finiteness property was introduced by J. Lagarias and Y. Wang, who conjectured in \cite{LW} that it is satisfied by every finite set of matrices. Theoretical results giving various preconditions for the finiteness property were established in \cite{LW,Gu2}. The existence of pairs  of matrices lacking the finiteness property was later demonstrated by T. Bousch and J. Mairesse in \cite{BM} (see also \cite{Koz3,BTV}). Exceptions to the finiteness property nonetheless seem to be rare: M. Maesumi conjectures in \cite{Mae} that the finiteness property holds for almost all finite sets of matrices in the sense of Lebesgue measure on $\mathcal{O}_r(\KK^d)$, and in a similar vein V. Blondel and R. Jungers conjecture in \cite{BJ} that the finiteness property holds for every finite set of rational matrices. In the first  theorem in this article we contribute to the study of the finiteness property by showing that for each $r,d \geq 2$, an open set of $r$-tuples of matrices in $\mathcal{O}_r(\KK^d)$ has the finiteness property.

We shall call an $r$-tuple $\mathcal{A} =(A_1,\ldots,A_r)\in \mathcal{O}_r(\KK^d)$ \emph{reducible} if there exists a linear space $V \subset \KK^d$, not equal to $\{0\}$ or $\mathbb{K}^d$, which is preserved by every $A_i$. Otherwise $\mathcal{A}$ shall be called \emph{irreducible}. Reducibility of $\mathcal{A}$ is equivalent to the existence of nonzero vectors $u,v \in \KK^d$ such that for each $n \in \NN$, $\left<A_{i_n}\cdots A_{i_1}u,v\right>=0$ for every $(i_1,\ldots,i_n)\in\{1,\ldots,r\}^n$. We let $\mathcal{I}_r(\KK^d)$ denote the set of irreducible elements of $\mathcal{O}_r(\KK^d)$. It is straightforward to show that $\mathcal{I}_r(\KK^d)$ is open and dense in $\mathcal{O}_r(\KK^d)$. 

A norm $\|\cdot\|$ on $\KK^d$ will be called \emph{extremal} for $\mathcal{A}=(A_1,\ldots,A_r)\in\mathcal{O}_r(\KK^d)$ if $\|A_i\| \leq \varrho(\mathcal{A})$ for every $i$. We shall say that $\|\cdot\|$ is a \emph{Barabanov norm} for $\mathcal{A}$ if the equation
\begin{equation}\label{Bnorm}
\varrho(\mathcal{A})\|v\| = \max_{1 \leq i \leq r}\|Av\|
\end{equation}
is satisfied for every $v \in \KK^d$, which in particular implies that $\|\cdot\|$ is extremal. If $\mathcal{A}$ is irreducible then $\varrho(\mathcal{A})$ is nonzero and a Barabanov norm for $\mathcal{A}$ exists \cite{Ba,Wirth1}. Clearly $\|\cdot\|$ is a Barabanov norm for $\mathcal{A}$ if and only if every norm proportional to $\|\cdot\|$ is also a Barabanov norm, so we shall simply say that $\mathcal{A}$ has a `unique' Barabanov norm to mean that all of the Barabanov norms for $\mathcal{A}$ are proportional to one another.
In the second theorem in this article we shall give a sufficient condition for the Barabanov norm of $\mathcal{A}$ to be unique.

To state our first theorem we require some further definitions. For each $r,n \in \NN$ we let $\Omega^n_r$ denote the set $\{1,\ldots,r\}^n$ , which we refer to as the set of \emph{words} of length $n$ over the alphabet $\{1,\ldots,r\}$. We shall say that two words $z=(z_1,\ldots,z_n),\omega=(\omega_1,\ldots, \omega_n) \in \Omega^n_r$ are \emph{rotation equivalent}, and write $z \sim \omega$, if there exists $k$ such that $(z_1,\ldots,z_n)= (\omega_{k+1},\ldots,\omega_n,\omega_1,\ldots,\omega_k)$. We write $\omega = z^p$, and say that $\omega$ is a power of $z$, if $\omega \in\Omega^{np}_r$ consists of $p$ repetitions of the word $z \in \Omega^n_r$. Clearly $\sim$ defines an equivalence relation on each $\Omega^n_r$, and if $(A_1,\ldots,A_n) \in \mathcal{O}_r(\KK^d)$ and $(z_1,\ldots,z_n)\sim (\omega_1,\ldots,\omega_n)$ then $\rho(A_{z_n}\cdots A_{z_1})=\rho(A_{\omega_n}\dots A_{\omega_1})$. In the terminology of \cite{BTV}, $z\sim \omega$ if and only if $z$ and $\omega$ have the same length and are `essentially equal'.

If $\mathcal{A} \in \mathcal{O}_r(\KK^d)$ and $\|\cdot\|$ is any matrix norm, then by a theorem of M. A. Berger and Y. Wang \cite{BW},
\begin{equation}\label{ZOMGBWF}\varrho(\mathcal{A}) = \inf_{n \in \NN} \max_{(i_1,\ldots,i_n)\in\Omega_r^n}\|A_{i_n} \cdots A_{i_1}\|^{1/n} = \sup_{n \in \NN} \max_{(i_1,\ldots,i_n)\in\Omega_r^n}\rho(A_{i_n} \cdots A_{i_1})^{1/n},\end{equation}
which in particular implies that $\varrho \colon \mathcal{O}_r(\KK^d) \to \RR$ is continuous for every $r,d \in \NN$.

Let us say that $\mathcal{A}=(A_1,\ldots,A_r) \in \mathcal{I}_r(\KK^d)$ satisfies the \emph{strong finiteness hypothesis} if there exists a word $\omega = (\omega_1,\ldots,\omega_n) \in \Omega_r^n$, which we call a \emph{characteristic word} for $\mathcal{A}$, such that for every Barabanov norm $\|\cdot\|_B$ for $\mathcal{A}$, we have $\|A_{i_n}\cdots A_{i_1}\|_B<\varrho(\mathcal{A})^n$ whenever $(i_1,\ldots,i_n) \nsim (\omega_1,\ldots,\omega_n)$. This condition implies that $\|(A_{\omega_n}\cdots A_{\omega_1})^m\|_B = \varrho(\mathcal{A})^{nm}$ for every $m \in \NN$, since if this were to fail for some $m$ then we would have
\[\max \left\{\left\|A_{i_{(m+1)n}}\cdots A_{i_1}\right\|_B \colon \left(i_1,\ldots,i_{(m+1)n}\right) \in \Omega^{(m+1)n}_r\right\}<\varrho(\mathcal{A})^{(m+1)n}\]
contradicting \eqref{ZOMGBWF}. It follows that if $\mathcal{A}$ satisfies the strong finiteness hypothesis with characteristic word $(\omega_1,\ldots,\omega_n)$ then it satisfies the finiteness property with $\varrho(\mathcal{A})=\rho(A_{\omega_n}\cdots A_{\omega_1})^{1/n}$. The reader may note that if $\omega \in \Omega^n_r$ is a characteristic word for $\mathcal{A}$, then $\hat\omega \in \Omega^n_r$ is a characteristic word if and only if $\hat\omega \sim \omega$, and for every $p \in \NN$, $\omega^p \in \Omega^{np}_r$ is also a characteristic word for $\mathcal{A}$.
The strong finiteness hypothesis has the following important property:
\begin{theorem}\label{FIST}
Let $r,d \in \NN$ and $\omega \in \Omega^n_r$. Then the set of all $\mathcal{A}\in \mathcal{I}_r(\KK^d)$ such that $\mathcal{A}$ satisfies the strong finiteness hypothesis with characteristic word $\omega$ is an open subset of $\mathcal{I}_r(\KK^d)$.\end{theorem}
The strong finiteness hypothesis appears to be problematic to verify, even in cases where $\mathcal{A}$ is otherwise easily analysed. Some illustrative examples are given in the following section. It is tempting to conjecture the following possible generalisation, which would be of broader practical value: 
\begin{conjecture}
Let $r,d \in \NN$ and $\hat \omega \in \Omega^n_r$, and let $\mathcal{U}_{\hat\omega} \subset \mathcal{I}_r(\KK^d)$ be the set of all $r$-tuples $\mathcal{A}=(A_1,\ldots,A_r)$ for which there exists an extremal norm $\|\cdot\|_{\mathcal{A}}$ with the property that $\|A_{\omega_n}\cdots A_{\omega_1}\|_{\mathcal{A}}<\varrho(\mathcal{A})^n$ whenever $\hat\omega \nsim (\omega_1,\ldots,\omega_n)$. Then $\mathcal{U}_{\hat\omega}$ is an open subset of $\mathcal{I}_r(\KK^d)$.
\end{conjecture}
We remark that the principal obstruction to such a result would appear to be the lack of a natural candidate extremal norm $\|\cdot\|_{\mathcal{B}}$ for $\mathcal{B} \in \mathcal{I}_r(\KK^d)$ close to $\mathcal{A}$. Indeed, the proof of Theorem \ref{FIST} functions by showing that if $\mathcal{A}$ satisfies the strong finiteness hypothesis, then every Barabanov norm for $\mathcal{B}$ is such a candidate norm.

While the strong finiteness hypothesis is quite a strict condition, we are nonetheless able to construct examples in which it is satisfied:
\begin{proposition}\label{TheProposition}
Let $r,n \in \NN$ and $\omega =(\omega_1,\ldots,\omega_n) \in \Omega_r^n$, and suppose that $\omega$ is not equal to a power of a shorter word. Then there exists $\mathcal{A} =(A_1,\ldots,A_r) \in \mathcal{I}_r(\KK^n)$ such that $\mathcal{A}$ satisfies the strong finiteness hypothesis with characteristic word $\omega$, the $A_i$ are pairwise distinct, and $\mathrm{rank}\,\left(A_{\omega_n}\cdots A_{\omega_1}\right)=1$.
\end{proposition}
An easy consequence of Theorem \ref{FIST} and Proposition \ref{TheProposition} is that for every $r,d \geq 2$, there exists a nonempty open subset of $\mathcal{O}_r(\KK^d)$ in which the finiteness property is everywhere satisfied. In view of Proposition \ref{TheProposition} it seems natural to ask the following question, which the author is not at present able to resolve: for fixed $r,d \geq 2$, is it the case that every $\omega \in \bigcup_{n =1}^\infty \Omega_r^n$ arises as the characteristic word of some $\mathcal{A} \in \mathcal{I}_r(\CC^d)$?

In order to state our second theorem we require some further definitions. Given $\mathcal{A}=(A_1,\ldots,A_r)\in\mathcal{O}_r(\KK^d)$ let us define
\[\mathcal{A}_n:=\left\{A_{i_n}\cdots A_{i_1}\colon (i_1,\ldots,i_n) \in \Omega^n_r\right\}\]
for each $n \in \mathbb{N}$. We shall say that $\mathcal{A}$ is \emph{product bounded} if the set $\bigcup_{n \in \mathbb{N}}\mathcal{A}_n$ is a bounded subset of $\Mat_d(\mathbb{K})$, and that $\mathcal{A}$ is \emph{relatively product bounded} if $\varrho(\mathcal{A})>0$ and $\varrho(\mathcal{A})^{-1}\mathcal{A}$ is product bounded. If $\varrho(\mathcal{A})>0$ and $\mathcal{A}$ admits an extremal norm then clearly $\mathcal{A}$ must be relatively product bounded, so in particular every irreducible $\mathcal{A}$ has this property. For relatively product bounded $\mathcal{A}$ the \emph{limit semigroup} of $\mathcal{A}$, introduced by F. Wirth in \cite{Wirth1}, is defined to be the set
\[\mathcal{S}(\mathcal{A}) := \bigcap_{m=1}^\infty \left(\overline{\bigcup_{n=m}^\infty \varrho(\mathcal{A})^{-n}\mathcal{A}_n}\right).\] 
In this article we shall say that $\mathcal{A} \in \mathcal{O}_r(\KK^d)$ has the \emph{rank one property} if it is relatively product bounded and every nonzero element of $\mathcal{S}(\mathcal{A})$ has rank one. By Proposition \ref{TheProposition}, every $\mathcal{I}_r(\KK^d)$ with $r,d \geq 2$ contains some $\mathcal{A}$ satisfying the rank one property. It turns out that the rank one property is also stable under small perturbations:
\begin{proposition}\label{R1P}
For each $d,r \in \NN$ the set of all $\mathcal{A} \in \mathcal{I}_r(\KK^d)$ satisfying the rank one property is open. In the particular case $\mathbb{K}=\mathbb{C}$, $d=2$ this set is also dense.
\end{proposition}
As an aside, we remark that there exist open subsets of $\mathcal{I}_2(\mathbb{R}^2)$ in which the rank one property does not hold. For example, let us consider the irreducible (with respect to $\RR^2$) pair of matrices $\mathcal{A}=(A_1,A_2)$ given by
\[A_1=\left(\begin{array}{cc}0&1\\1&0\end{array}\right), \qquad A_2=\left(\begin{array}{cc}\frac{1}{2}&0\\0&\frac{1}{2}\end{array}\right).\]
Since clearly $\vvv A_1\vvv=\rho(A_1)=1$ and $\vvv A_2 \vvv = \frac{1}{2}$ we have $\varrho(\mathcal{A})=\rho(A_1)=1$. If $\|\cdot\|$ is a Barabanov norm for $\mathcal{A}$ then $\|A_2\|=\frac{1}{2}<\varrho(\mathcal{A})$ since this matrix is a scalar multiple of the identity; it follows that $\mathcal{A}$ satisfies the strong finiteness hypothesis with $\omega = (1)$. By Theorem \ref{FIST} there exists a small open neighbourhood $U$ of $\mathcal{A}$ such that every $\mathcal{B}=(B_1,B_2)\in U$ is irreducible and satisfies $\varrho(\mathcal{B})=\rho(B_1)$. However, if $\mathcal{B}$ is close enough to $\mathcal{A}$ then $B_1$ must have a conjugate pair of eigenvalues, and it follows that $\mathcal{S}(\mathcal{B})$ contains the identity matrix when $\mathcal{B}$ is close enough to $\mathcal{A}$.

We shall say that $\mathcal{A} =(A_1,\ldots,A_r)\in \mathcal{O}_r(\KK^d)$ has the \emph{unbounded agreements property} if for any $N \in \NN$ and any pair of sequences $i_1,i_2 \colon \NN \to \{1,\ldots,r\}$ such that
\[\limsup_{n\to \infty} \left(\frac{\vvv A_{i_m(n)}\cdots A_{i_m(1)}\vvv}{\varrho(\mathcal{A})^n}\right)>0\]
for $m=1,2$, there exist $n_1,n_2 \in \NN$ such that $i(n_1+k)=j(n_2+k)$ throughout the range $1 \leq k \leq N$. It is straightforward to show that the strong finiteness hypothesis implies the unbounded agreements property, and thus Proposition \ref{TheProposition} implies that the unbounded agreements property is also satisfied in a nonempty subset of $\mathcal{I}_r(\KK^d)$ which contains an open set. On the other hand, the unbounded agreements property does not imply the finiteness property. If $\mathcal{A}$ denotes the pair of matrices shown in \cite{BM} to lack the finiteness property, then the existence of a unique height-maximizing shift-invariant measure on $\{0,1\}^{\NN}$ forces the unbounded agreements property to hold. We defer the proof of this assertion to a later publication.

We may now state our second theorem:
\begin{theorem}\label{UBN}
Suppose that $\mathcal{A} \in \mathcal{O}_r(\KK^d)$ is relatively product-bounded and satisfies both the unbounded agreements property and the rank one property. Then there exists at most one Barabanov norm for $\mathcal{A}$ up to multiplication by a scalar.
\end{theorem}


By combining all four of the results in this section one may easily obtain:
\begin{corollary}
For each $r,d \geq 2$ there exists a nonempty open set $U \subset \mathcal{I}_r(\KK^d)$ with the property that for every $\mathcal{A} \in U$, $\mathcal{A}$ satisfies the finiteness property, the unbounded agreements property and the rank one property, and has a unique Barabanov norm.
\end{corollary}
The remainder of this article is structured as follows. In the following section we give some examples of the applications and limitations of Theorems \ref{FIST} and \ref{UBN}, showing in particular that both the unbounded agreements property and the rank one property are necessary parts of the latter theorem. In \S\ref{poves} we give the proofs of Theorems \ref{FIST} and \ref{UBN} and Propositions \ref{TheProposition} and \ref{R1P}. Finally, in \S\ref{purple-aki} we briefly describe the relationship between Theorems \ref{FIST} and \ref{UBN} and certain recent results in ergodic theory.
\section{Examples}
In each of the examples below we may take $\mathbb{K}$ to be either $\mathbb{R}$ or $\mathbb{C}$. The reader may readily verify that each example is irreducible in both of the two cases.

{\bf{Example 1}}. Define
\[A_1=\left(\begin{array}{cc}0&1\\\lambda_1&0\end{array}\right),\qquad A_2=\left(\begin{array}{cc}0&\lambda_2\\1&0\end{array}\right),
\]
where $0 \leq |\lambda_1|,|\lambda_2|<1$. Then $\mathcal{A}=(A_1,A_2)$ satisfies the strong finiteness hypothesis, has the unbounded agreements property and the rank one property, and has a unique Barabanov norm. (In the case $\lambda_1=\lambda_2=0$ this is the example given by Proposition \ref{TheProposition} with $r=n=2$.)
\begin{proof}
A straightforward calculation yields $\vvv A_1A_2 \vvv = \vvv A_2A_1\vvv=\rho(A_1A_2)=1$, $\vvv A_1^2\vvv = |\lambda_1|<1$ and $\vvv A_2^2\vvv = |\lambda_2|<1$, from which it follows that $\varrho(\mathcal{A})=1$. If $i \colon \NN \to \{1,2\}$ has $\limsup_{n\to\infty} \vvv A_{i(n)}\cdots A_{i(1)}\vvv>0$ then it follows that there must exist $n_0>0$ such that $i(n_0+2n+1)=1$ and $i(n_0+2n)=2$ for every $n \in \NN$, which implies the unbounded agreements property. Since $|\det A_1|,|\det A_2|<1$ it is clear that every element of $\mathcal{S}(\mathcal{A})$ has determinant zero and hence has rank at most equal to one. The norm defined by $\|(x,y)^T\|:=\max\{|x|,|y|\}$ is a Barabanov norm for $\mathcal{A}$, since for each $v=(x,y)^T \in \KK^2$
\[\max \left\{\|A_1v\|,\|A_2v\|\right\} = \max\{|\lambda_1 x|,|y|,|\lambda_2 y|,|x|\}=\max\{|x|,|y|\}=\|v\|.\]
By Theorem \ref{UBN} it is the only such norm up to scalar multiplication. Direct calculation then shows that $\|A_1^2\|,\|A_2^2\|<1$ and hence the strong finiteness hypothesis holds.
\end{proof}

{\bf{Example 2}}. Define
\[A_1= \left(\begin{array}{cc}1&0\\0&\lambda\end{array}\right), \qquad A_2=\left(\begin{array}{cc}0&\lambda\\\lambda&0\end{array}\right),\]
where $0<|\lambda|<1$. Then $\mathcal{A}=(A_1,A_2)$ has the finiteness property, the unbounded agreements property and the rank one property, but does not satisfy the strong finiteness hypothesis.
\begin{proof}
Straightforward calculation shows that $\vvv A_1 \vvv = \rho(A_1)=1$ and $\vvv A_2 \vvv = |\lambda|<1$ so that in particular $\varrho(\mathcal{A})=1$ and $\mathcal{A}$ has the finiteness property. The unbounded agreements property and the rank one property follow in the same manner as for the previous example. Define a norm on $\KK^2$ by $\left\|(x,y)^T\right\|=\max\{| x|,|\lambda y|\}$. For each $v=(x,y)^T \in \KK^2$ we have
\[\max\{\|A_1v\|,\|A_2v\|\}=\max\left\{\left|x\right|,\left|\lambda^2y\right|,\left|\lambda y\right|,\left|\lambda^2x\right|\right\}=\max\{|x|,|\lambda y|\}=\|v\|\]
and therefore $\|\cdot  \|$ is a Barabanov norm for $\mathcal{A}$. Define $u:=(0,1)^T \in \KK^2$. For each $n \in \NN$ we have $\left\|A_1^{n-1}A_2u\right\|=|\lambda|=\|u\|$ and therefore $\left\|A_1^{n-1}A_2\right\|=1$, and clearly we also have $\left\|A_1^n\right\|=1$. We deduce that no $\omega \in \Omega_2^n$ can be a characteristic word for $\mathcal{A}$, and since $n$ is arbitrary we conclude that $\mathcal{A}$ does not satisfy the strong finiteness hypothesis.\end{proof}

{\bf{Example 3}}. Define
\[A_1=\left(\begin{array}{cc}1&0\\0&-1\end{array}\right), \qquad A_2=\left(\begin{array}{cc}0&\lambda\\\lambda&0\end{array}\right),\]
where $0<|\lambda|<1$. Then $\mathcal{A}=(A_1,A_2)$ has the unbounded agreements property and the finiteness property, but lacks the rank one property, and has more than one Barabanov norm.
\begin{proof}
 Clearly $\vvv A_1 \vvv = \rho(A_1)=1$ and $\vvv A_2 \vvv =|\lambda|<1$. It follows that $\varrho(\mathcal{A})=1$ and that $\mathcal{A}$ satisfies the finiteness property, and the unbounded agreements property follows as before. It is clear that $I \in \mathcal{S}(\mathcal{A})$ and hence $\mathcal{A}$ does not satisfy the rank one property. For any $p \geq 1$ the norm $\|\cdot\|_p$ defined by $\|(x,y)^T\|_p = (|x|^p+|y|^p)^{1/p}$ is a Barabanov norm since $\|A_2v\|_p \leq \|A_1v\|_p = \|v\|_p$ for all $v \in \mathbb{K}^2$.\end{proof}

{\bf{Example 4}}. Define
\[A_1 = \left(\begin{array}{cc}1&0\\0&0\end{array}\right),\qquad A_2=\left(\begin{array}{cc}0&0\\0&1\end{array}\right), \qquad A_3=\left(\begin{array}{cc}0&\lambda\\\lambda&0\end{array}\right),\]
where $0<|\lambda|<1$. Then $\mathcal{A}=(A_1,A_2,A_3)$ has the rank one property and the finiteness property, but lacks the unbounded agreements property, and has more than one Barabanov norm.
\begin{proof}
Clearly $\varrho(\mathcal{A})=1$, and $\mathcal{A}$ does not satisfy the unbounded agreements property since $\lim_{n\to\infty} A_1^n =  A_1 \neq 0$ but also $ \lim_{n\to\infty} A_2^n =A_2 \neq 0$. If $A \in \mathcal{A}_n$ with $\mathrm{rank}\,A=2$ then necessarily $A=A_3^n$ and hence $\vvv A\vvv = |\lambda|^n$. It follows that $\mathcal{S}(\mathcal{A})$ cannot contain a matrix of rank $2$. If $|\lambda|\leq\xi\leq |\lambda|^{-1}$ then the norm $\|\cdot\|_\xi$ defined by $\|(x,y)^T\|_\xi = \max\{|x|,\xi|y|\}$ is a Barabanov norm, since for every $v = (x,y)^T \in \mathbb{K}^2$
\begin{align*}
\max\left\{\|A_1v\|_\xi,\|A_2v\|_\xi,\|A_3v\|_\xi\right\} &=\max\left\{|x|,\xi|y|, |\lambda y|,\xi|\lambda x|\right\}\\ &= \max\left\{|x|,\xi|y|\right\} = \|v\|_\xi\end{align*}
as required.\end{proof}
\section{Proofs of Theorems \ref{FIST} and \ref{UBN} and Propositions \ref{TheProposition} and \ref{R1P}}\label{poves}

\subsection{Proof of Theorem \ref{FIST}}
Let $\mathcal{N}$ denote the set of all norms on $\KK^d$, and for $\|\cdot\|_1,\|\cdot\|_2 \in \mathcal{N}$ define
\[d_{\mathcal{N}}(\|\cdot\|_1,\|\cdot\|_2):= \sup\left\{\left|\log\left( \frac{\|v\|_1}{\|v\|_2}\right)\right|\colon \vvv v \vvv=1\right\}.\]
It is clear that $d_{\mathcal{N}}$ is a metric on $\mathcal{N}$. Moreover, if $\mathcal{Z} \subset \mathcal{N}$ is closed and has finite diameter with respect to $d_{\mathcal{N}}$, then it follows from the Arz\'ela-Ascoli theorem that $\mathcal{Z}$ is compact. We require the following lemma:
\begin{lemma}\label{wirthlemma}
Let $v_0 \in \KK^d$, let $\left(\mathcal{A}^{(k)}\right)_{k=1}^\infty$ be a sequence of elements of $\mathcal{I}_r(\KK^d)$ converging to some $\mathcal{A} \in \mathcal{I}_r(\KK^d)$, and for each $k$ let $\|\cdot\|_k$ be a Barabanov norm for $\mathcal{A}^{(k)}$ such that $\|v_0\|_k=1$. Then the sequence of norms $\|\cdot\|_{k}$ has an accumulation point in $\mathcal{N}$ which is a Barabanov norm for $\mathcal{A}$. 
\end{lemma}
\begin{proof}
Let $\mathcal{A}=(A_1,\ldots,A_r)$ and $\mathcal{A}^{(k)}=\left(A_1^{(k)},\ldots,A_r^{(k)}\right)$ for each $k \in \NN$. By hypothesis the set $\{\mathcal{A}^{(k)} \colon k \in \NN\}\cup\{\mathcal{A}\}$ is a compact subset of $\mathcal{I}_r(\KK^d)$, and it follows from \cite[Theorem 4.1]{Wirth1} that the set $\{\|\cdot\|_k \colon k \in \NN\}$ has finite diameter as a subset of $\mathcal{N}$. It follows that there exist a subsequence $(k_j)_{j=1}^\infty$ and norm $\|\cdot\|\in\mathcal{N}$ such that $\lim_{j \to \infty}\|\cdot\|_{k_j}=\|\cdot\|$. If $v \in \KK^d$, then for $i=1,\ldots,r$ and $k \in \NN$
\[\left|\left\|A_i^{(k_j)}v\right\|_{k} - \left\|A_i v\right\|_{k}\right|  \leq C \bigvvv A_i^{(k)} - A_i\bigvvv . \vvv v \vvv\]
and
\[ \left|\|A_iv\|_{k} - \|A_iv\|\right| \leq \|A_iv\|\left(e^{d_{\mathcal{N}}(\|\cdot\|_{k},\|\cdot\|)}-1\right),\]
implying that
\[\lim_{j \to \infty} \left\|A_i^{(k_j)}v\right\|_{k_j} = \|A_iv\|\]
for all such $i$ and $v$. We conclude that for every $v \in \KK^d$
\[\|v\| = \lim_{j \to \infty} \|v\|_{k_j} = \lim_{j \to \infty} \max_{1 \leq i \leq r} \left\|A_i^{(k_j)}v\right\|_{k_j} = \max_{1 \leq i \leq r} \left\|A_iv\right\|\]
and $\|\cdot\|$ is a Barabanov norm for $\mathcal{A}$.
\end{proof}
For each $\omega = (\omega_1,\ldots,\omega_n)$ and $\mathcal{A} =(A_1,\ldots,A_r)\in \mathcal{I}_r(\KK^d)$ let us now define
\begin{equation}\label{blegg}
\Theta_\omega(\mathcal{A}):= \sup \left\{\|A_{\omega_n}\cdots A_{\omega_1}\|^{1/n} \colon \|\cdot\|\text{ is a Barabanov norm for }\mathcal{A}\right\}.\end{equation}
Let $v_0 \in \KK^d \setminus \{0\}$ be arbitrary. If for each $k$ we choose $\|\cdot\|_k \in \mathcal{N}$ such that $\|A_{\omega_n}\cdots A_{\omega_1}\|^{1/n}_k > \Theta_\omega(\mathcal{A})-\frac{1}{k}$, then by rescaling each $\|\cdot\|_k$ so that $\|v_0\|_k=1$ and applying Lemma \ref{wirthlemma} with $\mathcal{A}^{(k)}\equiv \mathcal{A}$, we see that the supremum in \eqref{blegg} is always attained.

We claim that each $\Theta_\omega \colon \mathcal{I}_r(\KK^d) \to \RR$ is upper semi-continuous. For each $k \in \NN$ let $\mathcal{A}^{(k)} =\left(A^{(k)}_1,\ldots,A^{(k)}_r\right) \in \mathcal{I}_r(\KK^d)$ and suppose that $\lim_{k \to \infty}\mathcal{A}^{(k)} = \mathcal{A} = (A_1,\ldots,A_r) \in \mathcal{I}_r(\KK^d)$. Choose some arbitrary $v_0 \in \KK^d \setminus \{0\}$ and for each $k\in \NN$ let $\|\cdot\|_k$ be a Barabanov norm such that $\Theta_\omega(\mathcal{A}^{(k)})=\left\|A^{(k)}_{\omega_n}\cdots A^{(k)}_{\omega_1}\right\|_k$ and $\|v_0\|_k=1$. Choose a sequence of integers $(k_j)_{j=1}^\infty$ such that $\lim_{j\to\infty} \Theta_\omega(\mathcal{A}^{(k_j)}) = \limsup_{k\to\infty} \Theta_\omega(\mathcal{A}^{(k)})$. By Lemma \ref{wirthlemma}, we may replace $(k_j)$ with a finer subsequence such that $\lim_{j \to \infty}\|\cdot\|_{k_j} = \|\cdot\| \in \mathcal{N}$, where $\|\cdot\|$ is a Barabanov norm for $\mathcal{A}$. It follows that
\[\limsup_{k \to \infty}\Theta_\omega\left(\mathcal{A}^{(k)}\right) = \lim_{j \to 
\infty} \left\|A^{(k_j)}_{\omega_n} \cdots A^{(k_j)}_{\omega_1}\right\|_{k_j} = \left\|A_{\omega_n}\cdots A_{\omega_1}\right\| \leq \Theta_\omega(\mathcal{A}),\]
since $\|\cdot\|$ is a Barabanov norm for $\mathcal{A}$, which proves the claim.

Now let us consider some fixed $\hat\omega \in \Omega^n_r$. Since the supremum in \eqref{blegg} is attained for every $\mathcal{A}$ and $\omega$, it follows that $\mathcal{A} \in \mathcal{I}_r(\KK^d)$ satisfies the strong finiteness hypothesis with characteristic word $\hat\omega$ if and only if
\begin{equation}\label{botschafter}\max\left\{\Theta_\omega(\mathcal{A})\colon \omega \in \Omega^n_r \text{ and }\omega \nsim \hat\omega\right\}<\varrho(\mathcal{A})^n.\end{equation}
We have shown that each of the functions $\Theta_\omega$ is upper semi-continuous, and since $\varrho$ depends continuously on $\mathcal{A}$ it follows that the set of all $\mathcal{A} \in \mathcal{I}_r(\KK^d)$ solving the inequality \eqref{botschafter} must be open. The proof is complete.

\subsection{Proof of Proposition \ref{TheProposition}}

Let $r,n \in \NN$, and let $(\omega_1,\ldots,\omega_n)\in\Omega_r^n$. We shall begin by proving Proposition \ref{TheProposition} subject to the additional assumption that $(\omega_1,\ldots,\omega_n)$ includes at least one instance of every symbol $1,\ldots,r$. 

Let $e_1,\ldots,e_n$ be the standard basis for $\KK^n$, and let $\|\cdot\|$ be the norm on $\KK^n$ given by $\|\sum_{k=1}^n \lambda_k e_k\|=\max |\lambda_k|$. Define an $r$-tuple of matrices $\mathcal{A}=(A_1,\ldots,A_r)$ by setting $A_{\omega_i}e_i:=e_{i+1}$ for $1 \leq i <n$ and $A_{\omega_n}e_n:=e_1$, and $A_j e_k:=0$ in all other cases. Clearly $A_{\omega_n}\cdots A_{\omega_1}e_1=e_1$ and $\|A_i\| \leq 1$ for every $i$, and it follows that $\rho(A_{\omega_n}\cdots A_{\omega_1}) = \varrho(\mathcal{A})=1$. If $1 \leq i<j \leq r$ then by hypothesis there is $k$ such that $j=\omega_k \neq i$ and therefore $A_ie_k \neq A_je_k$, so in particular we have $A_i \neq A_j$ and the matrices forming $\mathcal{A}$ are pairwise distinct. 

We claim that if $(z_1,\ldots,z_n)\nsim (\omega_1,\ldots,\omega_n)$ then $A_{\omega_n}\cdots A_{\omega_1} =0$, which clearly implies that $(\omega_1,\ldots,\omega_n)$ must be a characteristic word for $\mathcal{A}$. Let $(z_1,\ldots,z_n) \in \Omega_r^n$ and suppose that $A_{z_n}\cdots A_{z_1} \neq 0$. To simplify notation in the remainder of this paragraph it will be convenient to add subscripts modulo $n$, identifying $n+1$ with $1$, $n+2$ with $2$, et cetera. Since $A_{z_n}\cdots A_{z_1}$ is nonzero there must exist $k$ such that $A_{z_n}\cdots A_{z_1}e_k \neq 0$. From the definition of $\mathcal{A}$ it follows that $z_1 = \omega_k$ and $A_{z_1}e_k = e_{k+1}$. We must then have  $A_{z_2}e_{k+1} \neq 0$ and by the same reasoning it follows that $z_2 = \omega_{k+1}$ and $A_{z_2}A_{z_1}e_k = A_{z_2}e_{k+1}=e_{k+2}$. Proceeding inductively in this fashion we obtain $A_{z_n}\cdots A_{z_1}e_k = e_k$ and $(z_1,\ldots,z_n)=(\omega_{k+1},\ldots,\omega_n,\omega_1,\ldots,\omega_{k})$, proving the claim. 

If $A_{\omega_n}\cdots A_{\omega_1}e_k \neq 0$ for some $k \neq 1$ then the above reasoning shows that $(\omega_1,\ldots,\omega_n)=(\omega_k,\ldots,\omega_n,\omega_1,\ldots,\omega_{k-1})$. It follows from this that  $(\omega_1,\ldots,\omega_n)$ is equal to $(\omega_1,\ldots,\omega_q)^p$ where $q=\mathrm{hcf}(k-1,n)$ and $n=pq$. Since $(\omega_1,\ldots,\omega_n)$ is by hypothesis not equal to a power of a shorter word we deduce that $A_{\omega_n}\cdots A_{\omega_1}$ has rank equal to one as claimed.

It remains to show that $\mathcal{A}$ is irreducible. Let $u,v \in \KK^d$ with $u \neq 0$, and suppose that for every $m \in \NN$ one has $\left<A_{i_m}\cdots A_{i_1}u,v\right>=0$ for every $(i_1,\ldots,i_m)\in \Omega^m_r$. We claim that necessarily $v=0$, implying the irreducibility of $\mathcal{A}$. Since $u \neq 0$ there is $k$ such that $\left<u,e_k\right>\neq 0$. Now, from the previous two paragraphs it follows that
\[A_{\omega_{k+1}}\cdots A_{\omega_1}A_{\omega_n}\cdots A_{\omega_k}u = \left<u,e_k\right>e_k\]
and therefore
\[A_{\omega_{n}}\cdots A_{\omega_1}A_{\omega_n}\cdots A_{\omega_k}u = \left<u,e_k\right>e_1.\]
We deduce in particular that $\left<e_1,v\right>=0$, and that for each $m \in \NN$ we must have  $\left<A_{i_m}\cdots A_{i_1}e_1,v\right>=0$ for every  $(i_1,\ldots,i_m)\in\Omega^m_r$. Since by construction $A_{\omega_{\ell}}\cdots A_{\omega_1}e_1 = e_{\ell+1}$ for $1 \leq \ell < n$, we conclude that $\left<e_j,v\right>$ must equal zero for every $j$, and it follows that $v=0$, proving the claim. This completes the proof of the theorem subject to our additional assumption.

The general case now follows easily. It suffices to assume that there is $1 \leq r' <r$ such that $(\omega_1,\ldots,\omega_n)$ includes at least one instance of every symbol $1,\ldots,r'$, since we can always get to and from this situation by perturbing the indices $i=1,\ldots,r$ in some manner. Let $(A_1,\ldots,A_{r'})$ be the $r'$-tuple of matrices produced by the preceding argument. To extend this to a full $r$-tuple we simply let $A_{r'+1},\ldots,A_r$ be given by small scalar multiples of the matrices already defined. If the modulus of each scalar is strictly less than one then it is clearly that this yields an irreducible $r$-tuple $\mathcal{A}$ with $\varrho(\mathcal{A})=1$. If $\|\cdot\|_{\mathcal{A}}$ is a Barabanov norm and $z=(z_1,\ldots,z_n) \in \Omega^n_r$ with $z \nsim \omega$, then either $\|A_{z_n}\cdots A_{z_1}\|_{\mathcal{A}}=0$ if $z_i \leq r'$ for every $i$, or $\|A_{z_n}\cdots A_{z_1}\|_{\mathcal{A}} \leq \prod_{i=1}^n \|A_{z_i}\|_{\mathcal{A}} <1=\varrho(\mathcal{A})^n$ if otherwise. The proof is complete.



\subsection{Proof of Proposition \ref{R1P}}
We shall use the following simple characterisation of the rank one property:
\begin{lemma}\label{wedgio}
Let $\mathcal{A} \in \mathcal{O}_r(\mathbb{K}^d)$ be relatively product bounded. Then $\mathcal{A}$ has the rank one property if and only if $\varrho(\wedge^2\mathcal{A})<\varrho(\mathcal{A})^2$, where $\wedge^2\mathcal{A}$ is the $r$-tuple which consists of the second exterior powers of the matrices comprising $\mathcal{A}$.
\end{lemma}
\begin{proof}
By normalising $\mathcal{A}$ if necessary it clearly suffices to consider the case $\varrho(\mathcal{A})=1$. Let $\|\cdot\|_{\wedge^2}$ denote the standard norm on $\mathbb{K}^d\wedge\mathbb{K}^d$. Recall that for every $A \in \Mat_d(\mathbb{K})$ we have
\begin{equation}\label{Bh}\left\|\wedge^2 A\right\|_{\wedge^2} = \vvv A\vvv.\inf\{\vvv A-B\vvv\colon \mathrm{rank}\,B \leq 1\}\end{equation}
see e.g. \cite{Bhatia}. In particular $\varrho(\wedge^2\mathcal{A})$ is at most $1$. 

Suppose that $\varrho(\wedge^2\mathcal{A})<1$. Given any nonzero matrix $A \in \mathcal{S}(\mathcal{A})$, choose an increasing sequence of integers $(n_k)_{k=1}^\infty$ and a sequence of matrices $(A_{n_k})_{k=1}^\infty$ such that $A_{n_k} \in \mathcal{A}_{n_k}$ for every $k$, and $\lim_{k \to \infty} A_{n_k} = A$. Since $\varrho(\wedge^2\mathcal{A})<1$ we have $\|\wedge^2 A\|_{\wedge^2} = \lim_{k\to\infty}\|\wedge^2 A_{n_k}\|_{\wedge^2} = 0$. It follows from \eqref{Bh} that $A$ has rank equal to $1$ and we conclude that the rank one property holds.

Suppose conversely that $\varrho(\wedge^2\mathcal{A})=1$. It follows that we can choose a sequence $(A_n)_{n=1}^\infty$ with each $A_n \in \mathcal{A}_n$ such that $\|\wedge^2A_n\|_{\wedge^2} \geq 1$, since if for any $n$ this choice were not possible we would have $\varrho(\wedge^2\mathcal{A})<1$ by \eqref{ZOMGBWF}. Since $\mathcal{A}$ is product bounded we may take a subsequence $(n_k)_{k=1}^\infty$ and matrix $A \in \mathcal{S}(\mathcal{A})$ such that $\lim_{k \to \infty}A_{n_k} = A$. Since $\|\wedge^2 A_n\|_{\wedge^2} \geq 1$ for every $n$ we have $\|\wedge^2 A\|_{\wedge^2}\geq 1$. By \eqref{Bh} it follows that the rank of $A$ is at least two, and we conclude that the rank one property does not hold.
\end{proof}
The proof of Proposition \ref{R1P} now follows easily. Since the maps from $\mathcal{I}_r(\KK^d)$ into $\RR$ given by $\mathcal{A} \mapsto \varrho(\mathcal{A})$ and $\mathcal{A} \mapsto \varrho(\wedge^2\mathcal{A})$ are both continuous, the set of all $\mathcal{A}\in\mathcal{I}_r(\KK^d)$ such that $\varrho(\wedge^2\mathcal{A})<\varrho(\mathcal{A})^2$ is open.

To see that the rank one property holds for a dense subset of $\mathcal{I}_r(\mathbb{C}^2)$ we argue as follows. If $\mathcal{A}=(A_1,\ldots,A_r)$ meets the condition that each $A_i$ has two distinct eigenvalues which are unequal in modulus, then
\[\varrho(\wedge^2\mathcal{A}) = \max_{1\leq i \leq r} |\det A_i| < \max_{1\leq i \leq r}\rho(A_i)^2 \leq \varrho(\mathcal{A})^2.\]
Since this condition is easily seen to be hold for a dense open subset of $\mathcal{O}_r(\CC^2)$, it follows that it holds for a dense open subset of $\mathcal{I}_r(\CC^d)$ also, which implies that the rank one property holds for this set.

\subsection{Proof of Theorem \ref{UBN}}

Let $\mathcal{A}=(A_1,\ldots,A_r) \in \mathcal{O}_r(\mathbb{K}^d)$ and suppose that $\mathcal{A}$ satisfies the unbounded agreements property. Without loss of generality we normalise $\mathcal{A}$ so as to obtain $\varrho(\mathcal{A})=1$. Let $\|\cdot\|_{B1}$ and $\|\cdot\|_{B2}$ be Barabanov norms for $\mathcal{A}$ which are not proportional to each other. Note that the existence of a Barabanov norm implies that $\mathcal{A}$ is product bounded. We will establish Theorem \ref{UBN} by showing that under these hypotheses the rank one property cannot hold.

Rescaling one of the two norms if necessary, there exists a real number $\lambda>1$ such that
\[\sup\{\|v\|_{B1}\colon \|v\|_{B2} = 1 \} = \sup\{\|v\|_{B2}\colon \|v\|_{B1}=1\} = \lambda.\]
Let us define subsets $X_1$, $X_2$ of $\mathbb{K}^d$ by
\[X_1 = \left\{v \in \mathbb{K}^d \colon \|v\|_{B1} = \lambda \|v\|_{B2} = \lambda\right\},\]
\[X_2 = \left\{v \in \mathbb{K}^d \colon \|v\|_{B2} = \lambda \|v\|_{B1} = \lambda\right\}.\]
Clearly $X_1$ and $X_2$ are nonempty and compact, and satisfy $X_1 \cap X_2 =\emptyset$. Note also that if $v_1 \in X_1$ and $v_2\in X_2$ then $v_1$ and $v_2$ are linearly independent.

To begin the proof we show that there exist sequences $i_1,i_2 \colon \NN \to \{1,\ldots,r\}$ and vectors $u_1 \in X_1$ and $u_2 \in X_2$ such that for every natural number $n$ we have $A_{i_m(n)}\cdots A_{i_m(1)}u_m \in X_m$ for $m=1,2$. Let $u_1$ and $u_2$ be arbitrary elements of $X_1$ and $X_2$ respectively. Since $\|\cdot\|_{B1}$ is a Barabanov norm there exists $i_1(1)\in \{1,\ldots,r\}$ such that $\|A_{i_1(1)}v_1\|_{B1}=\|v_1\|_{B1}=\lambda$ and consequently
\[\|A_{i_1(1)}v_1\|_{B1}  = \lambda \|v_1\|_{B2} = \lambda\max_{1\leq j \leq r} \|A_jv_1\|_{B2} \geq \lambda\|A_{i_1(1)}v_1\|_{B2} \geq  \|A_{i_1(1)}v_1\|_{B1}\]
which implies that $A_{i_1(1)}u_1 \in X_1$. Applying this procedure with the vector $A_{i_1(1)}u_1$ in place of $u_1$ allows us to define $i_1(2)$ with the property that $A_{i_1(2)}A_{i_1(1)}u_1 \in X_1$. Proceeding inductively in this manner we may thus construct the sequence $i_1$. The construction of the sequence $i_2$ may be undertaken in precisely the same fashion.

We next claim that for every $\ell \in \NN$ there exist a function $j_\ell \colon \{1,\ldots,\ell\} \to \{1,\ldots,r\}$ and two vectors $v_{\ell,1} \in X_1$, $v_{\ell,2}\in X_2$ such that for $1 \leq k \leq \ell$ and $m=1,2$ one has $A_{j_\ell(k)}\cdots A_{j_{\ell}(1)}v_{\ell,m} \in X_m$. To see this, let $i_1,i_2$ and $u_1,u_2$ be the sequences and vectors defined above. Since $\left\|A_{i_m(n)}\cdots A_{i_m(1)}u_m\right\|_{Bm}=\lambda$ for each $n \in \NN$ and for $m=1,2$, we conclude that neither of the two sequences of matrices $(A_{i_m(n)}\cdots A_{i_m(1)})_{n=1}^\infty$ converges to zero in the limit $n \to \infty$. Let $\ell \in \NN$. By the unbounded agreements property it follows that there exist $n_1,n_2 \in \NN$ such that $i_1(n_1+k)=i_2(n_2+k)$ throughout the range $1 \leq k \leq \ell$. Taking $j_{\ell}(k):=i_1(n_1+k)$ for $1 \leq k \leq \ell$ and setting $v_{\ell,m}=A_{i_m(n_m)}\cdots A_{i_m(1)}u_m \in X_m$ for $m=1,2$ proves the claim.

We may now complete the proof. For each $\ell \in \NN$ and $m=1,2$ we have $\left\|A_{j_\ell(\ell)}\cdots A_{j_\ell(1)}v_{\ell,m}\right\|_{Bm}=\|v_{\ell,m}\|_{Bm}$ and therefore $\left\|A_{j_\ell(\ell)}\cdots A_{j_\ell(1)}\right\|_{Bm}=1$. It follows that we may choose a subsequence $\left(\ell_n\right)_{n=1}^\infty$ such that the sequence of matrices $(B_n)_{n=1}^\infty$ given by $B_n:=A_{j_{\ell_n}(\ell_n)}\cdots A_{j_{\ell_n}(1)}$ converges to a nonzero matrix $B \in \Mat_d(\KK)$. By definition we have $B \in \mathcal{S}(\mathcal{A})$. Since each $X_m$ is compact, by taking further subsequences if necessary we may assume that $\lim_{n \to \infty}v_{\ell_n,m} = v_m \in X_m$ for $m=1,2$. For each $n \in \NN$ and $m=1,2$ we have $B_nv_{\ell_n,m} \in X_m$, and it follows from this that $B v_1 \in X_1$, $Bv_2 \in X_2$. But this implies that $Bv_1$ and $Bv_2$ are linearly independent, and we have obtained $B \in \mathcal{S}(\mathcal{A})$ with $\mathrm{rank}\,B \geq 2$. Thus $\mathcal{A}$ lacks the rank one property and the proof is complete.

\section{Connections with ergodic theory}\label{purple-aki}

The proof of Theorem \ref{UBN} is suggested by a lemma of T. Bousch in ergodic theory \cite[Lemme C]{B1}, which we shall briefly describe. Define a map $T \colon \RR / \ZZ \to \RR / \ZZ$ by $T(x)=2x (\mod 1)$ and let $f \colon \RR / \ZZ \to \RR$ be Lipschitz continuous. If we then define
\[\beta(f)=\inf_{n \in \NN} \sup_{x \in \RR / \ZZ}\frac{1}{n} \sum_{j=0}^{n-1} f(T^jx),\]
then Bousch's Lemme C gives criteria under which the functional equation
\begin{equation}\beta(f)+g(x)=\label{bush}\max_{T(y)=x} \left[f(y)+g(y)\right] \end{equation}
admits at most one continuous solution $g$ up to the addition of a real constant. The similarity between the functional equations \eqref{Bnorm} and \eqref{bush} was previously remarked on by Bousch in the manuscript \cite{Brect}. When more than one solution to \eqref{bush} exists, moreover, the set of all sufficiently regular solutions is an equicontinuous family \cite[Lemma 7.6]{CLOS}, a result suggestive of \cite[Theorem 4.1]{Wirth1}.

The idea of Theorem \ref{FIST} was suggested to the author by a theorem of Contreras, Lopes and Thieullen which is closely related to the work of Bousch described above \cite[Theorem 8]{CLT}. However, the proof of Theorem \ref{FIST} in its final form has no direct connection with the argument in \cite{CLT}. Connections between the joint spectral radius and optimisation problems in ergodic theory are also investigated by the author in \cite{MQBWF,MGBWF}.


\bibliographystyle{amsplain.bst}
\bibliography{Barrow}
\end{document}